\documentclass[11pt, twoside]{amsart}

\usepackage[colorlinks=true,citecolor=cyan,backref=page]{hyperref}

\hypersetup{nesting=true,debug=true,naturalnames=true}
\usepackage{graphicx,amssymb,bm}

\usepackage[colorlinks=true,citecolor=cyan,backref=page]{hyperref}

\usepackage[square,sort&compress,comma,numbers]{natbib}
\usepackage{nicefrac,xcolor,upref,version}

\usepackage{amscd,amsthm,amsmath,amssymb,amsfonts}

\usepackage{mathrsfs}

\usepackage{t1enc}

\newtheorem{theorem}{Theorem}[section]
\newtheorem{lemma}[theorem]{Lemma}

\newtheorem{remark}[theorem]{Remark} 
\newtheorem{definition}[theorem]{Definition}
\newtheorem{example}[theorem]{Example}

\newtheorem{question}[theorem]{Question}

\numberwithin{equation}{section}

\def\build#1_#2^#3{\mathrel{\mathop{\kern 0pt#1}\limits_{#2}^{#3}}}

\def\date {le\ {\the\day}\ \ifcase\month\or janvier
\or fevrier\or mars\or avril\or mai\or juin\or juillet\or
ao\^ut\or septembre\or octobre\or novembre
\or d\'ecembre\fi\ {\oldstyle\the\year}}

\font\fivegoth=eufm5 \font\sevengoth=eufm7 \font\tengoth=eufm10

\newfam\gothfam \scriptscriptfont\gothfam=\fivegoth
\textfont\gothfam=\tengoth \scriptfont\gothfam=\sevengoth
\def\goth{\fam\gothfam\tengoth}

\def\smallsquare{\vbox{\hrule\hbox{\vrule height 1 ex\kern 1 ex\vrule}\hrule}}

\def\eps{{\varepsilon}}
\def\ens{{\enspace}}

\def\Card{{\rm Card}}
 
\def\cDC{{\mathcal{DC}}}

\newcommand{\Q}{{\mathbb Q}}

\newcommand{\R}{{\mathbb R}}
\newcommand{\Z}{{\mathbb Z}}
\newcommand{\N}{{\mathbb N}}

\newcommand{\cK}{{\mathcal K}}

\newcommand{\Qbar}{\bar{\Q}}

\DeclareMathOperator{\Gal}{Gal}

\DeclareMathOperator{\Norm}{N}

\newcommand{\bfX}{{\mathbf X}}

\newcommand{\bfb}{{\mathbf b}}

\newcommand{\bmmu}{{\bm {\mu}}}

\newcommand{\cM}{\mathcal{M}}
\newcommand{\cN}{\mathcal{N}}

\newcommand{\cL}{\mathcal{L}}

\newcommand{\scrS}{\mathscr{S}}

\definecolor{myblue}{rgb}{0.6, 0.9, 1}
\definecolor{mygreen}{rgb}{0,0,1}
\definecolor{orange}{rgb}{0.8,0,0.2}

\begin{document}

\title[Arithmetical properties of convergents]{Some arithmetical properties of convergents to algebraic numbers}     

\author{ Yann Bugeaud} 
\address{Universit\'e de Strasbourg, IRMA UMR 7501, 7, rue Ren\'e Descartes, 67084 Strasbourg  (France)}
\address{Institut universitaire de France} 
\email{bugeaud@math.unistra.fr}

\author{Khoa D.~Nguyen}
\address{
Khoa D.~Nguyen \\
Department of Mathematics and Statistics\\
University of Calgary\\
AB T2N 1N4, Canada
}
\email{dangkhoa.nguyen@ucalgary.ca}

\bigskip
\begin{abstract} 
Let $\xi$ be an irrational algebraic  
real number and  $(p_k / q_k)_{k \ge 1}$ denote
the sequence of its convergents. 
Let $(u_n)_{n \geq 1}$ be a non-degenerate linear recurrence sequence of integers, which is not a
polynomial sequence. 
We show that if the intersection of the sequences $(q_k)_{k \ge 1}$ and $(u_n)_{n \geq 1}$ is infinite, then 
$\xi$ is a quadratic number. 
We also discuss several arithmetical properties of the sequence $(q_k)_{k \ge 1}$. 
\end{abstract}
\maketitle

\section{introduction}

Let $\theta$ be an arbitrary irrational 
real number and $(p_k (\theta) / q_k (\theta))_{k \ge 1}$ 
(we will use the shorter notation $p_k / q_k$ when no confusion is possible 
and $\xi$ instead of $\theta$
if the number is known to be algebraic) denote     
the sequence of its convergents.

Let $\cN$ be an infinite set of 
positive integers. It follows from a result of Borosh and Fraenkel \cite{BoFr72} that the set 
$$
\cK (\cN) = \{ \theta \in \R : \hbox{$q_k (\theta)$ is in $\cN$ for arbitrarily large $k$} \}
$$
has always Hausdorff dimension at least $1/2$ and its Lebesgue measure is zero if there is some 
positive $\delta$ such that the series $\sum_{q \in \cN} q^{-1 + \delta}$ converges. 
Examples of sets $\cN$ (or integer sequences $(u_n)_{n \ge 1}$) with the latter property include 
non-degenerate linear recurrence sequences, the set of integers 
having a bounded number of nonzero digits in their base-$10$ representation, 
sets of positive values taken at integer values by a given integer polynomial of degree at least $2$, and 
sets of positive integers divisible only by prime numbers from a given, finite set.

Our main purpose is to discuss whether $\cK (\cN)$ contains algebraic numbers for 
some special sets $\cN$ for which $\cK (\cN)$ has zero Lebesgue measure. 
Said differently, for an arbitrary irrational  
real algebraic number $\xi$, we investigate 
various arithmetical properties of the 
sequence 
$(q_k (\xi))_{k \ge 1}$. We consider the following questions:
\begin{itemize}
\item [A.] Does the greatest prime factor of $q_k (\xi)$ tends to infinity with $n$? If yes, how rapidly? 

\item [B.] Does the number of nonzero digits in the 
base-$10$   
representation of $q_k (\xi)$ tends to infinity with $n$? If yes, how rapidly? 

\item [C.] Are there infinitely many squares (cubes, perfect powers) in $(q_k (\xi))_{k \ge 1}$?

\item [D.] Is the intersection of $(q_k (\xi))_{k \ge 1}$ with a given linear recurrence sequence of integers finite or infinite?    
\end{itemize}

First, let us recall that very few is known on the continued fraction expansion of an algebraic number of degree at least $3$, while 
the continued fraction expansion of a quadratic real number $\xi$ is ultimately periodic and takes the form 
$$
\xi  = [a_0 ; a_1, \ldots , a_r, \overline{a_{r+1}, \ldots , a_{r+s}}]. 
$$
Consequently, we have $q_{k +2s} = t q_{k+s} - (-1)^{s} q_k$, for $k > r$, where $t$ is the trace of 
$$
\begin{pmatrix} a_{r+1} & 1 \\ 1 & 0 \end{pmatrix} \begin{pmatrix} a_{r+2} & 1 \\ 1 & 0 \end{pmatrix} \ldots 
\begin{pmatrix}a_{r+s} & 1 \\ 1 & 0 \end{pmatrix};
$$
see \cite{JaLi88,Ki82}. This shows that $(q_k (\xi))_{k \ge 1}$ is the union of $s$ binary recurrences whose roots are the 
roots of the polynomial $X^2 - t X + (-1)^s$, that is, the real numbers $(t \pm \sqrt{t^2 - 4 (-1)^s})/2$. 
Thus, for a quadratic real number $\xi$, we immediately derive Diophantine results on $(q_k (\xi))_{k \ge 1}$ 
from results on binary recurrences of the above form.

Question A has already been discussed in \cite{Bu09} and earlier works. Let us mention that it easily follows from Ridout's theorem \cite{Rid57} that the 
greatest prime factor of $q_k (\xi)$ tends to infinity with $n$, but we have no estimate of the rate of growth, except 
when $\xi$ is quadratic (by known effective results on binary recurrences, see \cite{Ste13b}).       
Furthermore, the theory of linear forms in logarithms gives a lower bound for the greatest prime 
factor of the product $p_k (\xi) q_k (\xi)$, which tends to infinity at least as fast as 
some constant times $\log_2 q_k(\xi) \log_3 q_k (\xi) / \log_4 q_k(\xi)$, where $\log_j$ denotes the 
$j$-th iterated logarithm function.  Although we have no new  contribution to Question A 
as stated for algebraic numbers $\xi$, we obtain new results on prime factors of $q_k(\theta)$   
for a transcendental number $\theta$. In 1939, Erd\H os and Mahler \cite{ErMa39} 
proved that the greatest prime factor 
of $q_{k-1}(\theta)q_k(\theta)q_{k+1}(\theta)$
tends to infinity as $k$ tends to infinity. In this paper, we obtain a more explicit result involving 
the irrationality exponent of $\theta$.

Question B has not been investigated up to now and we show how the $p$-adic Schmidt Subspace Theorem 
allows us to make some contribution. Question C is solved when $\xi$ is quadratic: 
there are only finitely many perfect powers in the sequence $(q_k(\xi))_{k \ge 1}$ thanks to results of Peth\H o \cite{Pet82} 
and Shorey and Stewart \cite{ShSt83} stating that there are only finitely many perfect 
powers in binary recurrence sequences of integers. This result is effective. When $\xi$ has degree at least $3$,  Question C appears to be very difficult. Since $(n^d)_{n\geq 1}$ is a linear recurrence sequence 
for any given positive integer $d$, a large part of Question C is contained in Question D.

Question D is interesting for several reasons. 
First, some assumption on the linear recurrence must be added, since the linear 
recurrence $(n)_{n \ge 1}$ has infinite intersection with the sequence $(q_k (\xi))_{k \ge 1}$. 
Second, as already mentionned, when $\xi$ is quadratic, its continued fraction expansion is ultimately 
periodic and the sequence $(q_k (\xi))_{k \ge 1}$ is the union of a finite set 
of binary recurrences. 
Among our results, we show that if a ``non-singular'' linear recurrence has an infinite intersection with 
$(q_k (\xi))_{k \ge 1}$, then $\xi$ must be quadratic. 
Unfortunately our results exclude sequences of the form $(n^d)_{n\geq 1}$, hence we 
do not have any contribution to Question C.

Recall that any non-zero linear recurrence sequence $(u_n)_{n\geq 1}$ of complex numbers can be expressed as
$$
u_n=P_1(n)\alpha_1^n+\cdots+P_r(n)\alpha_r^n\ \text{for $n\geq 1$},
$$
where $r \geq 1$, $\alpha_1,\ldots,\alpha_r$ are distinct non-zero complex numbers
(called the roots of the recurrence), and $P_1,\ldots,P_r$ 
are non-zero polynomials with complex coefficients. This expression is unique up to rearranging the terms. 
The sequence $(u_n)_{n\geq 1}$ is called non-degenerate if $\alpha_i/\alpha_j$ is not a root of unity for $1\leq i\neq j\leq r$. 
For most problems about linear recurrence sequences, it is harmless to assume that $(u_n)_{n\geq 1}$ is non-degenerate.  
Indeed, if $(u_n)_{n\geq 1}$ is degenerate and $L$ denotes 
the lcm of the orders of the roots of unity of the form $\alpha_i/\alpha_j$, then each of the subsequences 
$(u_{nL+m})_{n\geq 1}$ with $m\in\{0,\ldots,L-1\}$ is either identically zero or non-degenerate.

Let $(u_n)_{n \ge 1}$ be a non-constant linear recurrence sequence with integral roots 
greater than $1$ and rational coefficients. 
It follows from \cite[Theorem 4.16]{CZ18}  that the intersection of the 
sequences $(u_n)_{n \ge 1}$ and $(q_k)_{k \ge 1}$ is finite. 
This gives a first partial result toward Question D. 
We extend it to more general linear recurrence sequences: 

\begin{theorem}   \label{th:Main}
Let $(p_k/q_k)_{k \ge 1}$ be the sequence of convergents to a real algebraic number $\xi$ of degree at least $3$. 
Let $(u_n)_{n \geq 1}$ be a non-degenerate linear recurrence sequence of integers, which is not a
polynomial sequence. 
Then there are only finitely many pairs $(n,k)$ such that $u_n=q_k$.
\end{theorem}

The above theorem follows immediately from Theorem~\ref{thm:main} below. 
This stronger theorem uses the results
of Kulkarni, Mavraki, and Nguyen \cite{KMN19}, 
which extend a seminal work of Corvaja and Zannier \cite{CZ04}. For a real number $\theta$, we let
$\Vert\theta\Vert$ denote the distance from $\theta$ to the nearest integer.

\begin{theorem}\label{thm:main}
Let $\xi$ be an irrational algebraic number of degree $d$. 
Let $\varepsilon>0$. Let $(u_n)_{n \geq 1}$ be a 
non-degenerate 
recurrence sequence of integers, which is not a polynomial sequence.  
Then the set
$$
\left\{n\in\N:\ u_n\neq 0\ \text{and}\ \Vert u_n\xi\Vert < \frac{1}{\vert u_n\vert^{(1/(d-1))+\varepsilon}}\right\}
$$ 
is finite.
\end{theorem}

The case $d=2$ of Theorem \ref{thm:main} is immediate, since quadratic real numbers have bounded partial 
quotients in their continued fraction expansion. Consequently, we restrict our attention to the case $d \ge 3$. 
Theorem \ref{thm:main} is a special case of Theorem \ref{thm:1/2}, which deals with a larger class of integer 
sequences than that of recurrence sequences. 

The exponent $1/(d-1)$ in Theorem \ref{thm:main} is stronger than what we really need since any exponent $< 1$ 
would be sufficient for our purpose of proving Theorem~\ref{th:Main}
(recall that any convergent $r/s$ to $\xi$ satisfies $|\xi - r/s| < 1/s^2$).    
When $d=3$, this exponent $1/(d-1)=1/2$ is best possible, as can be seen with the following example.
Let $K\subset\mathbb{R}$ be a cubic field with a pair of complex-conjugate embeddings. 
Let $\xi\in K$ with $\vert\xi\vert>1$ be a unit of the ring of integers. Let $\alpha$ and $\bar{\alpha}$ denote the remaining Galois 
conjugates of $\xi$. We have $\vert\alpha\vert=\vert\xi\vert^{-1/2}$ and, 
setting $u_n=\xi^n+\alpha^n+\bar{\alpha}^n$ for $n \ge 1$, we check that
$$
\vert u_n\xi-u_{n+1}\vert \ll_\xi \vert \alpha^n\vert \ll_\xi \vert u_n\vert^{-1/2}, \quad \hbox{for $n \ge 1$},
$$ 
where $\ll_\xi$ means that the implicit constant is positive and depends only on $\xi$.    

When $d\geq 4$, we do not know if Theorem~\ref{thm:main} remains valid with a smaller exponent than $1/(d-1)$.

Theorem \ref{th:Main} allows us to complement  
the following result of Lenstra and Shallit \cite{LeSh93}. 

\begin{theorem}[Lenstra and Shallit]  \label{LeSh}
Let $\theta$ be an irrational real number, whose continued fraction
expansion is given by $\theta = [a_0; a_1, a_2, \ldots]$, and let $(p_k)_{k \ge 1}$ and $(q_k)_{k \ge 1}$ 
be the sequence of
numerators and denominators of the convergents to $\theta$. Then
the following four conditions are equivalent:
\begin{itemize}
\item [(i)] The sequence $(p_k)_{k \ge 1}$ satisfies a linear recurrence with constant complex coefficients;

\item [(ii)] The sequence $(q_k)_{k \ge 1}$ satisfies a linear recurrence with constant complex coefficients;

\item [(iii)] The sequence  $(a_n)_{n \ge 0}$ is ultimately periodic;

\item [(iv)] $\theta$ is a quadratic irrational.
\end{itemize}
\end{theorem}

The proof of Theorem \ref{LeSh} rests on the Hadamard Quotient Theorem. 
A simpler proof of a more general statement has been given by
B\'ezivin \cite{Be94}, who instead of (ii) only assumes that 
$(q_k)_{k \ge 1}$ satisfies a linear recurrence with coefficients being polynomials in $k$ 
and that the series $\sum_{k \ge 1} q_k z^k$ has a nonzero convergence radius.

We strengthen Theorem \ref{LeSh} for convergents of algebraic numbers as follows. 
\begin{theorem}  \label{LeShbis}
Let $\xi = [a_0; a_1, a_2, \ldots]$ be an irrational real algebraic number, and let $(p_k)_{k \ge 1}$ and $(q_k)_{k \ge 1}$ 
be the sequence of
numerators and denominators of the convergents to $\xi$. Then
the following four conditions are equivalent:
\begin{itemize}
\item [(i)] The sequence $(p_k)_{k \ge 1}$ has an infinite intersection with 
some non-degenerate linear recurrence sequence 
that is not a polynomial sequence;

\item [(ii)] The sequence $(q_k)_{k \ge 1}$ has an infinite intersection with 
some non-degenerate linear recurrence sequence
that is not a polynomial sequence;

\item [(iii)] The sequence  $(a_n)_{n \ge 0}$ is ultimately periodic;

\item [(iv)] $\xi$ is a quadratic irrational.
\end{itemize}
\end{theorem}

Now we present our results concerning Question B. Let $b \ge 2$ be an integer. Every positive integer $N$ can be written 
uniquely as 
$$
N = d_k b^{k} + \ldots + d_1 b + b_0, 
$$
where
$$
d_0, d_1, \ldots , d_k \in \{0, 1, \ldots , b-1\}, \quad
d_k \not= 0.
$$
We define the length 
$$
\cL (N, b) = \Card \{ 0 \le j \le k : d_j \not= 0 \}
$$ 
of the $b$-ary representation of $N$. 
We also define the number of digit changes by 
$$
\cDC (N, b) = \Card \{ 2 \le j \le k : d_j \not= d_{j-1} \}. 
$$

\begin{theorem}\label{thm:divisor & length}
Let $\xi$ be an irrational real algebraic number and let $b\geq 2$ be an integer. Let $(u_n)_{n\geq 1}$ be a strictly  increasing sequence of 
positive integers and  $\lambda\in (0,1]$ such that for every $\varepsilon>0$, the inequality
$$\Vert u_n\xi\Vert<  u_n^{-\lambda+\varepsilon}$$     
holds for all but finitely many $n$. We have:
\begin{itemize}
    \item [(i)] Let $k$ be a positive integer and let $\varepsilon>0$. For all sufficiently large $n$, if $\delta$ is a divisor of $u_n$ with $\cL(\delta,b)\leq k$ then $\delta<u_n^{(k-\lambda)/k+\varepsilon}$.
    \item [(ii)] Let $k$ be a non-negative integer and let $\varepsilon>0$. For all sufficiently large $n$, if $\delta$ is a divisor of $u_n$ with $\cDC(\delta,b)\leq k$ then  $\delta<u_n^{(k+2-\lambda)/(k+2)+\varepsilon}$.      
\end{itemize}
Consequently, let $(p_k/q_k)_{k\geq 1}$ denote the sequence of convergents to $\xi$ then each one of the limits 
$\displaystyle\lim_{k \to + \infty} \, \cL (q_k, b)$,
$\displaystyle\lim_{k \to + \infty} \, \cDC (q_k, b)$,
$\displaystyle\lim_{k \to + \infty} \, \cL (p_k, b)$, and 
$\displaystyle\lim_{k \to + \infty} \, \cDC (p_k, b)$
is infinite. 
\end{theorem}

Except for certain quadratic numbers, it seems to be a very difficult problem 
to get an effective version of the last assertion of Theorem \ref{thm:divisor & length}. 
Stewart \cite[Theorem 2]{Ste80} established that
if $(u_n)_{n \ge 1}$ is a binary sequence of integers, whose roots $\xi, \xi'$ are quadratic numbers 
with $|\xi| > \max\{1, |\xi'|\}$, then there exists a positive real number $C$ such that   
$$
\cL (u_n, b) > \frac{\log n}{\log \log n + C} - 1, \quad n \ge 5.   
$$
Consequently, if $(p_k / q_k)_{k \ge 1}$ denote the sequence of convergents to a quadratic 
real algebraic number,
then for $k \ge 4$ we have    
$$
\cL (q_k, b) >  \frac{\log k}{\log \log k + C} - 1
\quad \hbox{and} \quad
\cDC (q_k, b) >  \frac{\log k}{\log \log k + C} - 1. 
$$

A similar question can be asked for the Zeckendorf representation \cite{Zec72}    
of $q_k$. 
Let $(F_n)_{n \ge 0}$ denote the Fibonacci sequence defined by $F_0 = 0$, 
$F_1 = 1$, and $F_{n +2} = F_{n+1} + F_n$ for $n \ge 0$. Every positive 
integer $N$ can be written uniquely as a sum
$$
N = \eps_\ell F_\ell + \eps_{\ell -1} F_{\ell - 1} + \ldots + \eps_2 F_2 + \eps_1 F_1,   
$$
with $\eps_\ell = 1$, $\eps_j$ in $\{0, 1\}$, and 
$\eps_j \eps_{j+1} = 0$ for $j =1, \ldots , \ell-1$. 
This representation of $N$ is called its Zeckendorf representation. 
The number of digits of $N$ in its Zeckendorf representation is the number of positive integers $j$ 
for which $\eps_j$ is equal to $1$. 
By using the Schmidt Subspace Theorem we can in a similar way   
prove that 
the number of digits of $q_k(\xi)$ in its Zeckendorf representation tends to infinity with $k$, we omit the 
details (but see \cite{Bu21}).

Our last result is motivated by a theorem of Erd\H os and Mahler \cite{ErMa39} on convergents to 
real numbers. 
Let $S$ be a set of prime numbers. For a non-zero integer $N$, let $[N]_S$ denote the largest divisor of $N$ 
composed solely of primes from $S$. Set $[0]_S = 0$.   
Recall that the irrationality exponent $\mu(\theta)$ 
of an irrational real number $\theta$ is the supremum of the real numbers $\mu$ such that   
there exist infinitely many rational numbers $r/s$ with $s \ge 1$ and $|\theta - r/s| < 1 / s^\mu$.  
It is always at least equal to $2$ and, by definition, $\theta$ is called a Liouville number  
when $\mu (\theta)$ is infinite. 
Erd\H os and Mahler \cite{ErMa39} established that, when $\theta$ is irrational and not a Liouville number, 
then the greatest prime factor of   
$q_{k-1} q_k q_{k+1}$ tends to infinity with $k$. 
We obtain the following more precise version of their result. 

\begin{theorem}   \label{EM}
Let $\theta$ be an irrational real number and $\mu$ its irrationality exponent. 
Let $(p_k/q_k)_{k \ge 1}$ denote the sequence of convergents to $\theta$. 
Let $S$ be a finite set of prime numbers.
If $\mu$ is finite, then, for every $\eps > 0$ and every $k$ sufficiently large (depending on $\eps$), we have
\begin{equation}\label{eq:EM for q}
[q_{k-1} q_k q_{k+1}]_S < (q_{k-1} q_k q_{k+1})^{\mu/(\mu + 1) + \eps}.    
\end{equation}
The same conclusion holds when  the sequence 
$(q_k)_{k\geq 1}$ is replaced by 
$(\vert p_k\vert)_{k\geq 1}$.
\end{theorem}

When $\theta$ is algebraic irrational and $\eps > 0$, we have
$[q_k]_S<q_k^{\varepsilon}$ for all large $k$ by Ridout's theorem.    
The interesting feature of Theorem~\ref{EM} is that it holds for all transcendental  non-Liouville 
numbers.

Theorem \ref{EM} is ineffective. Under its assumption, it is proved in \cite{Bu23} that there exists a (large)  
positive, effectively computable $c = c(S)$ such that 
$$
[q_{k-1} q_k q_{k+1}]_S < (q_{k-1} q_k q_{k+1})^{1 - 1 / (c \mu \log \mu)}, \quad k \ge 2. 
$$

For $\mu = 2$ (that is, for almost all $\theta$), the exponent in \eqref{eq:EM for q} becomes $2/3 + \eps$.    
It is an interesting question to determine whether it is best possible. It cannot be smaller than $1/3$.  
Indeed, the Folding Lemma (see e.g. \cite[Section 7.6]{Bu12}) allows one, for any given 
integer $b \ge 2$, to construct explicitely real numbers $\theta$ with $\mu (\theta) = 2$ and having 
infinitely many convergents whose denominator is a power of $b$.

Furthermore, there exist irrational real numbers $\theta = [a_0; a_1, a_2, \ldots]$ with convergents 
$p_k / q_k$ such that the $q_k$'s are alternating among powers of $2$ and $3$. 
Indeed, let $k \ge 2$ and assume that $q_{k-1}=2^c$ and $q_k=3^d$ for positive integers $c, d$. 
Then, we have to find a positive integer $a_{k+1}$ such that 
$2^c + a_{k+1}3^d$ is a power of $2$.  To do this, it is sufficient to take for $m$ the smallest integer greater than $c$  
such that  $2^{m-c}$ is congruent to $1$ modulo $3^d$ and then define $a_{k+1} = (2^{m-c} - 1) / 3^d$. 
The sequence $(a_k)_{k \ge 1}$ increases very fast and $\theta$ is a Liouville number. 

The proof of Theorem \ref{EM} and additional remarks on \cite{ErMa39} are given in Section \ref{sec:4}.  
Theorem~\ref{thm:divisor & length} is established in Section \ref{sec:2} and the other results 
are proved in Section \ref{sec:3}.

\section{Proof of Theorem~\ref{thm:divisor & length}}   \label{sec:2} 

For a prime number $\ell$, we let $v_{\ell}:\ \Q\to \Z\cup\{\infty\}$ be the additive $\ell$-adic valuation and let 
$\vert \cdot\vert_\ell = \ell^{-v_\ell (\cdot)}$ be the $\ell$-adic absolute value.

\begin{proof}[Proof of Theorem \ref{thm:divisor & length}]
First, we prove part (i). 
Let $\cN_1$ be the set of tuples 
$(m, n_1, \ldots ,n_a)$ such that: 
\begin{itemize}
\item $1\leq a\leq k$ and $n_1<n_2<\ldots<n_a$ are non-negative integers. 

\item There exist $d_1,\ldots,d_a$ in $\{1,\ldots,b-1\}$
such that $\delta:= d_a b^{n_a} + \cdots + d_1 b^{n_1}$ 
is a divisor of $u_m$ and $\delta\geq u_m^{(k-\lambda)/k+\varepsilon}$.
\end{itemize}

Assume that $\cN_1$ is infinite. 
Then, there exist an integer $h$ with $1 \le h \le k$, positive integers $D_1, \ldots , D_h$,
an infinite set $\cN_2$ of $(h+1)$-tuples 
$(m_i, n_{1,i}, \ldots , n_{h,i})$ for $i\geq 1$ such that:
\begin{itemize}
\item $n_{1,i} < \ldots < n_{h,i}$ are non-negative intergers.
\item For $i\geq 1$, $\delta_{m_i}:= D_h b^{n_{h,i}} + \cdots + D_1 b^{n_{1,i}}$ is a divisor of $u_{m_i}$
with $\delta_{m_i}\geq u_{m_i}^{(k-\lambda)/k+\varepsilon}$.
\item We have
\begin{equation}\label{eq:nji-nj-1i}
\lim_{i \to + \infty} \, (n_{j,i} - n_{j-1,i}) = + \infty, \quad j = 2, \ldots , h.  
\end{equation}
\end{itemize}

For $i\geq 1$, let $w_{m_i}$ denote the nearest integer to $u_{m_i}\xi$ and let 
\begin{equation}\label{eq:vmi}
v_{m_i}:=u_{m_i}/\delta_{m_i}\leq u_m^{\lambda/k-\varepsilon}.
\end{equation}
When $m_i$ is sufficiently large, we have:
\begin{equation}\label{eq:um and lambda}
\vert\xi D_h v_{m_i}b^{n_{h,i}}+\cdots+\xi D_1v_{m_i}b^{n_{1,i}}-w_{m_i}\vert=\Vert\xi u_{m_i}\Vert<\vert u_{m_i}\vert^{-\lambda+\varepsilon/2}
\end{equation}
thanks to the given properties of $(u_m)_{m\geq 1}$ and $\lambda$.
We are in position to apply the Schmidt Subspace Theorem. 

Let $S$ denote the set of prime divisors of $b$. 
Consider the linear forms in $\bfX = (X_0, X_1, \ldots , X_h)$ given by 
$$
L_{j, \infty} (\bfX) := X_j, \quad  j = 1, \ldots , h, \quad
$$
$$
L_{0, \infty} (\bfX) := \xi D_h X_h + \ldots + \xi D_1 X_1 - X_0, 
$$
and, for every prime number $\ell$ in $S$,
$$
 L_{j, \ell} (\bfX) := X_j, \quad  j = 0, \ldots , h. 
$$

For the tuple $\bfb_i = (w_{m_i}, v_{m_i}b^{n_{h,i}}, \ldots , v_{m_i}b^{n_{2,i}}, v_{m_i}b^{n_{1,i}})$ with a sufficiently large $m_i$, we use \eqref{eq:vmi} and \eqref{eq:um and lambda} to obtain 
\begin{align*}
\prod_{j=0}^h \, |L_{j, \infty} (\bfb_i)| 
\times \prod_{\ell \in S } \,
\prod_{j=0}^h \, |L_{j, \ell} (\bfb_i)|_{\ell}&\leq\Vert\xi u_{m_i}\Vert \cdot \vert v_{m_i}\vert^h\\
&< \vert u_{m_i}\vert^{-(h-1/2)\varepsilon} \ll H(\bfb_i)^{-(h-1/2)\varepsilon}, 
\end{align*}
where the implied constant is independent of $i$ and $H(\bfb_i)$ is the Weil height of the projective point $[w_{m_i}:v_{m_i}b^{n_{h,i}}:\ldots:v_{m_i}b^{n_{1,i}}]$.

The Subspace Theorem \cite[Corollary~7.2.5]{BG06}
implies that there exist integers $t_0, t_1, \ldots , t_h$, not all zero, and an infinite
subset $\cN_3$ of $\cN_2$ 
such that  
\begin{equation}\label{eq:the t's}
v_{m_i}(t_h b^{n_{h,i}} + \cdots + t_1 b^{n_{1,i}}) + t_0 w_{m_i} = 0\quad \text{for $(w_{m_i},n_{h,i},\ldots,n_{1,i})\in\cN_3$}.
\end{equation}
Dividing the above equation by $u_{m_i}$ and letting $i$ tend to infinity, we deduce that 
$$
\frac{t_h}{D_h} + t_0 \xi = 0.
$$
Since $\xi$ is irrational, we must have $t_0 = t_h = 0$. 
Then, we use  \eqref{eq:nji-nj-1i} and \eqref{eq:the t's} to derive that $t_1 = \ldots = t_{h-1} = 0$, a contradiction. This finishes the proof of (i).   

\medskip

We now prove part (ii) using a similar method. Let $s\geq 0$ and let $x$ be a positive integer such that $\cDC(x,b)=s$. If $s=0$, we can write
$$x=d+db+\cdots+db^n=\frac{db^{n+1}-d}{b-1}$$
with $n\geq 0$ and $d\in\{1,\ldots,b-1\}$. If $s>0$, let $0<c_1<c_2<\ldots<c_s$ denote the exponents of $b$ where digit changes take place:
\begin{align*}
x&=d_0(1+\cdots+b^{c_1-1})+d_1(b^{c_1}+\cdots+b^{c_2-1})+\cdots+d_s(b^{c_s}+\cdots+b^n)\\
&=\frac{-d_0+(d_0-d_1)b^{c_1}+(d_1-d_2)b^{c_2}+\cdots+(d_{s-1}-d_s)b^{c_s}+d_sb^{n+1}}{b-1}
\end{align*}
with $n\geq c_s$, $d_0,\ldots,d_s\in \{0,\ldots,b-1\}$, and $d_{i+1}\neq d_i$ for $0\leq i\leq s-1$. 

Let $\cN_4$ be the set of tuples $(m,n_0,n_1,\ldots,n_a)$ such that:
\begin{itemize}
\item $0\leq a\leq k+1$ and $n_0<\ldots<n_a$ are non-negative integers.
\item There exist integers $e_0,\ldots,e_a$ in $[-(b-1),b-1]$ such that $\delta:=\displaystyle\frac{e_0b^{n_0}+\cdots+e_{k+1}b^{n_{k+1}}}{b-1}$
is a divisor of $u_m$ and  $\delta\geq u_m^{(k+2-\lambda)/(k+2)+\varepsilon}$.
\end{itemize}

Assume that $\cN_4$ is infinite. Then, there exist an integer $h$ with $0 \le h \le k+1$, non-zero integers $E_0, \ldots , E_h$,
an infinite set $\cN_5$ of $(h+2)$-tuples 
$(m_i, n_{h,i}, \ldots , n_{0,i})$ for $i\geq 1$ such that:
\begin{itemize}
\item $n_{0,i}<\ldots<n_{h,i}$ are non-negative integers. \item For $i\geq 1$, $\delta_{m_i}:=\displaystyle \frac{E_hb^{n_{h,i}}+\cdots+E_0b^{n_{0,i}}}{b-1}$
is a divisor of $u_{m_i}$ with $\delta_{m_i}\geq u_{m_i}^{(k+2-\lambda)/(k+2)+\varepsilon}$.
\item We have
\begin{equation*}
\lim_{i \to + \infty} \, (n_{j,i} - n_{j-1,i}) = + \infty, \quad j = 1, \ldots , h.  
\end{equation*}
\end{itemize}
We can now apply the Subspace Theorem in essentially the same way as before to finish the proof.
\end{proof}

\section{Proof of Theorem \ref{th:Main}, Theorem \ref{thm:main}, and Theorem \ref{LeShbis}}   \label{sec:3} 
Theorem~\ref{th:Main} follows from Theorem~\ref{thm:main} since $\Vert q_k\xi\Vert<\vert q_k\vert^{-1}$.
In Theorem~\ref{LeShbis}, the equivalence (iii) $\Leftrightarrow$ (iv) and the implications 
(iv) $\Rightarrow$ (i) and (iv) $\Rightarrow$ (ii) are well-known and have already appeared in Theorem~\ref{LeSh}. The implication (ii) $\Rightarrow$ (iv) is essentially Theorem~\ref{th:Main} while
the remaining implication (i) $\Rightarrow$ (iv) follows from the inequality
$\Vert p_k/\xi\Vert\ll_\xi \vert p_k\vert^{-1}$  
and Theorem~\ref{thm:main}. We spend the rest of this section to discuss Theorem~\ref{thm:main}.
	
From now on $\N$ is the set of positive integers, $\N_0=\N\cup\{0\}$, $\bmmu$ is the group of roots of unity, and $G_{\Q}=\Gal(\Qbar/\Q)$. Let $h$ denote the absolute logarithmic Weil height on $\Qbar$. Let $k\in\N$, a tuple $(\alpha_1,\ldots,\alpha_k)$ of non-zero complex numbers is called non-degenerate if
$\alpha_i/\alpha_j\notin \bmmu$ for $1\leq i\neq j\leq k$.   
 We consider the following more general family of sequences than (non-degenerate) linear recurrence sequences:

\begin{definition} \label{def:S(K)}
Let $K$ be a number field. Let $\scrS(K)$ be the set of all sequences $(u_n)_{n \geq 1}$ 
of complex numbers with the following property. There exist $k\in \N_0$
 together with a non-degenerate tuple $(\alpha_1,\ldots,\alpha_k)\in (K^*)^k$
such that, when $n$ is sufficiently large, we can express
\begin{equation}\label{eq:S(K)}
u_n=q_{n,1}\alpha_1^n+\cdots+q_{n,k}\alpha_k^n
\end{equation}
for $q_{n,1},\ldots,q_{n,k}\in K^*$
and $\max_{1\leq i\leq k} h(q_{n,i}) = o(n)$. 
\end{definition} 

In Definition~\ref{def:S(K)}, we allow $k=0$ for which the empty sum in the RHS of \eqref{eq:S(K)} means $0$. 
Any sequence $(u_n)_{n \geq 1}$ that is eventually  $0$ is in $\scrS(K)$.

\begin{example}
Consider a linear recurrence sequence $(v_n)_{n \geq 1}$ of the form:
$$v_n=P_1(n)r_1^n+\cdots+P_k(n)r_k^n$$
with $k\in\N$, distinct $r_1,\ldots,r_k\in K^*$, and non-zero $P_1,\ldots,P_k\in K[X]$.
Let $L$ be the lcm of the order of the roots of unity that appear among the 
$r_i/r_j$ for $1\leq i,j\leq k$. Then each one of the $L$ sequences
$(v_{nL+r})_{n \geq 1}$
for $r=0,\ldots,L-1$ is a member of $\scrS(K)$.

As an explicit example, consider $v_n=2^n+(-2)^n+n$ for $n\in\N$. The sequence $(v_{2n}=2\cdot 4^n+2n)_{n \geq 1}$ is in $\scrS(\Q)$ and a tuple $(\alpha_1,\ldots,\alpha_k)$ satisfying the requirement in Definition~\ref{def:S(K)} is 
$(\alpha_1=4,\alpha_2=1)$. The sequence $(v_{2n+1}=2n+1)_{n \geq 1}$ is in $\scrS(\Q)$ and a tuple $(\alpha_1,\ldots,\alpha_k)$ satisfying the requirement in Definition~\ref{def:S(K)} is 
$(\alpha_1=1)$. 
\end{example}

\begin{lemma}\label{lem:2 data}
Let $K$ be a number field and let  $(u_n)_{n \geq 1}$ be an element of $\scrS(K)$.
Let $k,\ell\in\N_0$ and let $(\alpha_1,\ldots,\alpha_k)$ and $(\beta_1,\ldots,\beta_{\ell})$ be non-degenerate tuples of non-zero elements of $K$. Suppose that when $n$ is sufficiently large, we can express
$$u_n=q_{n,1}\alpha_1^n+\cdots+q_{n,k}\alpha_k^n=r_{n,1}\beta_1^n+\cdots+r_{n,\ell}\beta_{\ell}^n$$
for $q_{n,1},\ldots,r_{n,\ell}\in K^*$ such that 
$$
\max\{h(q_{n,i}),h(r_{n,j}):\ 1\leq i\leq k,\ 1\leq j\leq \ell\}=o(n)
$$ 
as $n$ tends to infinity.
Then $k=\ell$ and there exist a permutation $\sigma$ of $\{1,\ldots,k\}$ 
together with roots of unity $\zeta_1,\ldots,\zeta_k$ in $K$ such that
$\alpha_{i}=\zeta_i \beta_{\sigma(i)}$ for $1\leq i\leq k$ and 
$q_{n,i}\zeta_i^n=r_{n,\sigma(i)}$ for every sufficiently large $n$ and for every $1\leq i\leq k$.
\end{lemma}
\begin{proof}
This follows from \cite[Proposition 2.2]{KMN19}. 
\end{proof}

\begin{definition}
Let $K$ be a number field and let $(u_n)_{n \geq 1}$ be in $\scrS(K)$. Let $(\alpha_1,\ldots,\alpha_k)$ satisfy the requirement in Definition~\ref{def:S(K)}. We call $k$ the number of $\scrS(K)$-roots of $(u_n)_{n \geq 1}$; this is well-defined thanks to Lemma~\ref{lem:2 data}. We call $(\alpha_1,\ldots,\alpha_k)$ a tuple of $\scrS(K)$-roots of $(u_n)_{n \geq 1}$; this is well-defined up to permuting the $\alpha_i$'s and multiplying each $\alpha_i$ by a root of unity in $K$. 
\end{definition}

Here is the reason why we use the strange terminology ``$\scrS(K)$-roots'' instead of the usual ``characteristic roots''. In the theory of linear recurrence sequences, we have the well-defined notion of characteristic roots. For example, \emph{the} characteristic roots of $(u_n=2^n+1)_{n \geq 1}$ are $2$ and $1$. When regarding $(u_n)_{n \geq 1}$ as an element of $\scrS(K)$, we may say that any tuple $(2\zeta,\zeta')$
where $\zeta$ and $\zeta'$ are roots of unity in $K$ is a tuple of $\scrS(K)$-roots of $(u_n)_{n \geq 1}$.

\begin{definition}\label{def:admissible scrS}
Let $K$ be a number field. Let $(u_n)_{n \geq 1}$ be an element of $\scrS(K)$ and let $k\in\N_0$  be its number of $\scrS(K)$-roots. We say that $(u_n)_{n \geq 1}$ is admissible if:
\begin{itemize}
	\item either $k=0$ i.e. $(u_n)_{n \geq 1}$ is eventually $0$
	\item or $k>0$ and at least one entry in a tuple of $\scrS(K)$-roots of $(u_n)_{n \geq 1}$ is not a root of unity.
\end{itemize}  
\end{definition}

Since every non-degenerate linear recurrence sequence of algebraic numbers that is not a polynomial sequence is an admissible element of $\scrS(K)$ for some number field $K$, Theorem~\ref{thm:main} follows from the below theorem.

\begin{theorem}\label{thm:1/2}
Let $\xi$ be an algebraic number of degree $d\geq 3$.
Let $\varepsilon>0$ and let $K$ be a number field. Let $(u_n)_{n \geq 1}$ be a sequence of integers that is also an admissible element of $\scrS(K)$. Then the set
$$\left\{n\in\N:\ u_n\neq 0\ \text{and}\ \Vert u_n\xi\Vert < \frac{1}{\vert u_n\vert^{(1/(d-1))+\varepsilon}}\right\}$$ 
is finite.
\end{theorem}

By a sublinear function, we mean a function $f:\ \N\rightarrow (0,\infty)$ 
such that $\displaystyle\lim_{n \to \infty} f(n)/n =0$, that is, $f(n) = o(n)$. The following theorem is established in \cite{KMN19}. 

\begin{theorem}\label{thm:KMN}
	Let $K$ be a number field, let $k\in\N$,
	let $(\alpha_1,\ldots,\alpha_k)$ be a non-degenerate
	tuple of algebraic numbers satisfying
	$\vert\alpha_i\vert \geq 1$
	for $1\leq i\leq k$, and let $f$ be a sublinear function.
	Assume that for some $\theta\in (0,1)$,
	the set $\cM$ of $(n,q_1,\ldots,q_k)\in \N\times (K^*)^{k}$ satisfying:
	$$\left\Vert \sum_{i=1}^k q_i\alpha_i^n\right\Vert <\theta^n\ \text{and}\ \max_{1\leq i\leq k} h(q_i)<f(n)$$
	is infinite. Then
\begin{itemize} 
		\item [(i)] $\alpha_i$ is an algebraic integer for
		$i=1,\ldots,k$.
		\item [(ii)] For each $\sigma\in G_{\Q}$ and $i=1,\ldots,k$ such that $\displaystyle\frac{\sigma(\alpha_i)}{\alpha_j}\notin \bmmu$ for $j=1,\ldots,k$ we have
		$\vert \sigma(\alpha_i)\vert<1$. 
\end{itemize}
Moreover, for all but finitely
	many $(n,q_1,\ldots,q_k)\in \cM$ we have: for $(\sigma,i,j)\in G_{\Q}\times \{1,\ldots,k\}^2$, $\sigma(q_i\alpha_i^n)=q_j\alpha_j^n$ if and only if $\displaystyle \frac{\sigma(\alpha_i)}{\alpha_j}\in\bmmu$.
\end{theorem}		

For the proof of Theorem~\ref{thm:1/2}, we need a slightly more flexible version:
\begin{theorem}\label{thm:KMN with C}
Let $C\in (0,1]$. Let $K$ be a number field, let $k\in\N$,
	let $(\alpha_1,\ldots,\alpha_k)$ be a non-degenerate
	tuple of algebraic numbers satisfying
	$\vert\alpha_i\vert \geq C$
	for $1\leq i\leq k$, and let $f$ be a sublinear function.
	Assume that for some $\theta\in (0,C)$,
	the set $\cM$ of $(n,q_1,\ldots,q_k)\in \N\times (K^*)^{k}$ satisfying:
	$$\left\Vert \sum_{i=1}^k q_i\alpha_i^n\right\Vert <\theta^n\ \text{and}\ \max_{1\leq i\leq k} h(q_i)<f(n)$$
	is infinite. Then
\begin{itemize} 
		\item [(i)] $\alpha_i$ is an algebraic integer for
		$i=1,\ldots,k$.
		\item [(ii)] For each $\sigma\in G_{\Q}$ and $i=1,\ldots,k$ such that $\displaystyle\frac{\sigma(\alpha_i)}{\alpha_j}\notin \bmmu$ for $j=1,\ldots,k$ we have
		$\vert \sigma(\alpha_i)\vert<C$. 
\end{itemize}
Moreover, for all but finitely
	many $(n,q_1,\ldots,q_k)\in \cM$ we have: for $(\sigma,i,j)\in G_{\Q}\times \{1,\ldots,k\}^2$, $\sigma(q_i\alpha_i^n)=q_j\alpha_j^n$ if and only if $\displaystyle \frac{\sigma(\alpha_i)}{\alpha_j}\in\bmmu$.
\end{theorem}
	
\begin{remark}
Theorem~\ref{thm:KMN with C} in the case $C=1$ is exactly Theorem~\ref{thm:KMN}.
\end{remark}	

\begin{proof}[Proof of Theorem~\ref{thm:KMN with C}]
When $n$ is fixed, there are only finitely many tuples 
$(n,q_1,\ldots,q_k)$ in $\cM$ thanks to the upper bound on $\max h(q_i)$ 
and Northcott's property.    
In the following, for $(n,q_1,\ldots,q_k)$ in $\cM$, we tacitly assume that $n$ is sufficiently large.



For $N$ large enough, we have $1 / \theta^N > 3 / C^N$ and the interval $[1/C^N,1/\theta^N)$ contains 
a prime number $D$ which does not divide the denominator of $\alpha_i$ for $i=1, \ldots , k$. We have
$$D\theta^N<1\leq DC^N.$$
Fix $\theta'\in (D\theta^N,1)$. Let $\beta_i=D\alpha_i^N$ for $1\leq i\leq k$. We now define the set $\cM'$ as follows. 
Consider $(n,q_1,\ldots,q_k)\in\cM$ with $n\equiv r$ mod $N$, write $n=mN+r$ with $r\in \{0,\ldots,N-1\}$, then we have
$$
\left\Vert\sum_{i=1}^k q_i\alpha_i^r \beta_i^m\right\Vert=\left\Vert\sum_{i=1}^k q_i\alpha_i^r(D\alpha_i^N)^m\right\Vert< D^m\theta^n=\theta^r (D\theta^N)^m<\theta'^m,
$$
assuming $n$ and hence $m$ are sufficiently large so that 
the last inequality holds thanks to the choice $\theta'\in (D \theta^n,1)$.  
We include the tuple $(m,q_1\alpha_1^r,\ldots,q_k\alpha_k^r)$ in $\cM'$. Finally, consider the sublinear function
$$
g(n)=f(n)+(N-1)\max_{1\leq i\leq k}h(\alpha_i),
$$
so that $\max_{1\leq i\leq k} h(q_i\alpha_i^r)<g(n)$.

We apply Theorem~\ref{thm:KMN} for the tuple $(\beta_1,\ldots,\beta_k)$, the function $g$, the number $\theta'$, and the set $\cM'$ to conclude that:
\begin{itemize}
\item $D\alpha_i^N$ is an algebraic integer for $1\leq i\leq k$. 
Our choice of $D$ implies that 
$\alpha_i$ is an algebraic integer for $1\leq i\leq k$. 
\item For each $\sigma\in G_{\Q}$ and $i\in\{1,\ldots,k\}$ such that $\displaystyle\frac{\sigma(\alpha_i)}{\sigma(\alpha_j)}\notin \bmmu$ for every $j\in\{1,\ldots,k\}$, we have $\sigma(D\alpha_i^N)<1$ consequently
$\sigma(\alpha_i)<1/D^{1/N}\leq C$.
\item The last assertion of Theorem~\ref{thm:KMN with C} holds.
\end{itemize}
This finishes the proof.
\end{proof}


%

\begin{proof}[Proof of Theorem~\ref{thm:1/2}]
Let $k$ denote the number of $\scrS(K)$-roots of $(u_n)_{n \geq 1}$. The case $k=0$ (i.e. $(u_n)_{n \geq 1}$ is eventually $0$) is obvious. Assume $k>0$ and let $(\alpha_1,\ldots,\alpha_k)$ be 
a tuple of $\scrS(K)$-roots of $(u_n)_{n \geq 1}$. For $L\in\N$ and $r\in \{0,\ldots,L-1\}$, each  sequence $(u_{nL+r})_{n \geq 1}$ is an admissible element of $\scrS(K)$ and admits $(\alpha_1^L,\ldots,\alpha_k^L)$ as a tuple of $\scrS(K)$-roots. Let $L$ be the lcm of the order of roots of unity among the
$\sigma(\alpha_i)/\tau(\alpha_j)$ for $\sigma,\tau\in G_{\Q}$ and $1\leq i,j\leq k$ and replace $(u_n)_{n \geq 1}$ by each $(u_{nL+r})_{n \geq 1}$, we may assume:
\begin{equation}\label{eq:reduction}
\text{for $\sigma,\tau\in G_{\Q}$ and $1\leq i,j\leq k$, $\frac{\sigma(\alpha_i)}{\tau(\alpha_j)}\in\bmmu$ iff $\sigma(\alpha_i)=\tau(\alpha_j)$.}
\end{equation}
We first prove that the set $\{\alpha_1,\ldots,\alpha_k\}$ is Galois invariant.

For sufficiently large $n$, express
$$u_n=q_{n,1}\alpha_1^n+\cdots+q_{n,k}\alpha_k^n$$      
as in Definition~\ref{def:S(K)}. Let $\sigma\in G_{\Q}$, since $u_n\in\Z$ we have:
$$
q_{n,1}\alpha_1^n+\cdots+q_{n,k}\alpha_k^n=\sigma(q_{n,1})\sigma(\alpha_1)^n+\cdots+\sigma(q_{n,k})\sigma(\alpha_k)^n
$$
for all large $n$. From \cite[Proposition 2.2]{KMN19}, we have that for every $i\in\{1,\ldots,k\}$ there exists
$j\in\{1,\ldots,k\}$ such that $\sigma(\alpha_i)/\alpha_j\in\bmmu$ and this gives
$\sigma(\alpha_i)=\alpha_j$ thanks to \eqref{eq:reduction}. 
Theorem~\ref{thm:KMN with C} implies that the $\alpha_i$'s are algebraic integers and  for every sufficiently large $n$, for $(\sigma,i,j)\in G_{\Q}\times \{1,\ldots,k\}^2$ we have: 
\begin{equation}\label{eq:sigma(qni)=qnj}
\sigma(q_{n,i})=q_{n,j}\ \text{whenever}\ \sigma(\alpha_i)=\alpha_j.
\end{equation}
Since
$(u_n)_{n \geq 1}$ is admissible, at least one of the $\alpha_i$'s is not a root of unity and hence 
\begin{equation}\label{eq:M>1}
M:=\max_{1\leq i\leq k}\vert \alpha_i\vert > 1.
\end{equation}

Suppose the set 
$$T:=\left\{n\in\N:\ u_n\neq 0\ \text{and}\ \Vert u_n\xi\Vert < \frac{1}{\vert u_n\vert^{(1/(d-1))+\varepsilon}}\right\}$$
is infinite then we will arrive at a contradiction.
Let $\delta$ denote a sufficiently small positive real number that will be specified later. By  
\cite[Section~2]{KMN19}, we have:
\begin{equation}
\vert u_n\vert > M^{(1-\delta)n}
\end{equation}
for all large $n$. Therefore
\begin{equation}\label{eq:(1-delta)(1/2+epsilon)}
\Vert \xi q_{n,1}\alpha_1^n+\cdots+\xi q_{n,k}\alpha_k^n\Vert < \frac{1}{M^{(1-\delta)(1/(d-1)+\varepsilon)n}}
\end{equation}
for all large $n$ in $T$.

We relabel the $\alpha_i$'s and let $m\leq \ell\leq k$ such that:
\begin{itemize}
\item [(i)] $\vert\alpha_1\vert=M$.
\item [(ii)] $\displaystyle\vert \alpha_i\vert \geq \frac{1}{M^{1/(d-1)+\delta}}$ for $1\leq i\leq \ell$ while
$\displaystyle\vert \alpha_i\vert <\frac{1}{M^{1/(d-1)+\delta}}$ for $\ell+1\leq i\leq k$. 
\item [(iii)] Among the $\alpha_1,\ldots,\alpha_\ell$,     
we have that $\alpha_1,\ldots,\alpha_m$ are exactly the Galois conjugates of $\alpha_1$. When combining with (ii), this means that $\alpha_1,\ldots,\alpha_m$ are precisely the Galois conjugates of $\alpha_1$ with modulus at least $M^{-(1/(d-1)+\delta)}$.
\end{itemize}

We require $\delta$ small enough so that:
\begin{equation}\label{eq:delta ineq 1}
1/(d-1)+\delta<(1-\delta)(1/(d-1)+\varepsilon).
\end{equation}

Choose the real number $c$ such that:
\begin{align}\label{eq:c ineq}
\begin{split}
1/(d-1)+\delta<c<(1-\delta)(1/(d-1)+\varepsilon)\\
\text{and}\ \vert\alpha_i\vert<\frac{1}{M^c}\ \text{for $\ell+1\leq i\leq k$.}
\end{split}
\end{align}

Thanks to this choice of $c$ and the assumption that $h(q_{n,i})=o(n)$ for $1\leq i\leq k$, we have
\begin{equation}
\vert \xi q_{n,\ell+1}\alpha_{\ell+1}^n+\cdots+\xi q_{n,k}\alpha_{k}^n\vert <\frac{1}{2M^{cn}},
\end{equation}
for all sufficiently large $n$. From \eqref{eq:(1-delta)(1/2+epsilon)} and \eqref{eq:c ineq}, we have
\begin{equation}
\Vert \xi q_{n,1}\alpha_1^n+\cdots+\xi q_{n,k}\alpha_k^n\Vert < \frac{1}{2M^{cn}},
\end{equation}
for all large $n$ in $T$. Combining the above inequalities, we have
\begin{equation}\label{eq:key to use KMN}
\Vert \xi q_{n,1}\alpha_1^n+\cdots+\xi q_{n,\ell}\alpha_{\ell}^n\Vert < \frac{1}{M^{cn}},
\end{equation}
for all large $n$ in $T$.

Let $F$ be the Galois closure of $K(\xi)$. We apply Theorem~\ref{thm:KMN with C} for the tuple
$(\alpha_1,\ldots,\alpha_{\ell})$, 
 $C=M^{-(1/(d-1)+\delta)}$, $\theta=M^{-c}$, and the inequality \eqref{eq:key to use KMN} then use 
 \eqref{eq:reduction} and \eqref{eq:sigma(qni)=qnj} 
  to have that for every large $n$ in $T$, $\sigma\in \Gal(F/\Q)$,
 and $1\leq i,j\leq \ell$,
 \begin{equation}\label{eq:sigma(xi)=xi}
 \text{if $\sigma(\alpha_i)=\alpha_j$ then $\sigma(\xi q_{n,i}\alpha_i^n)=\xi q_{n,j}\alpha_j^n$
 and hence $\sigma(\xi)=\xi$.} 
 \end{equation}
Since $\alpha_1,\ldots,\alpha_m$ are exactly the Galois conjugates of $\alpha_1$ among the $\alpha_1,\ldots,\alpha_{\ell}$, 
\eqref{eq:sigma(xi)=xi} implies that $\xi$ is fixed 
by at least $m\vert \Gal(F/\Q(\alpha_1))\vert=m[F:\Q(\alpha_1)]$ many automorphisms in $\Gal(F/\Q)$. 
Put $d'=[\Q(\alpha_1):\Q]$, we have
\begin{equation}\label{eq:[F:Q(xi)]}
[F:\Q(\xi)]=\vert \Gal(F/\Q(\xi))\vert\geq m[F:\Q(\alpha_1)] =\frac{m}{d'} [F:\Q].
\end{equation}

Since $[\Q(\xi):\Q]=d$, \eqref{eq:[F:Q(xi)]} implies
$m\leq d'/d$. This means $\alpha_1$ has at least $d'(d-1)/d$ many Galois conjugates with modulus less than
$M^{-(1/(d-1)+\delta)}$. Combining with the fact that all Galois conjugates of $\alpha_1$ have modulus at most $M$, we have:
$$
\vert \Norm_{\Q(\alpha_1)/\Q}(\alpha_1)\vert\leq M^{d'/d} M^{-(1/(d-1)+\delta)d'(d-1)/d}<1,
$$
since $M>1$ and $\delta > 0$.  
This contradicts the fact that $\alpha_1$ is a non-zero algebraic integer and we finish the proof.
\end{proof}

\section{Proof of Theorem~\ref{EM} and further discussion on the work of Erd\H os and Mahler.}   \label{sec:4}

\begin{proof}[Proof of Theorem~\ref{EM}]
We assume that $\theta$ is not a Liouville number, that is, we assume that $\mu$ is finite. Define
$$
Q_k = q_{k-1} q_k q_{k+1}, \quad k \ge 2. 
$$
Let $S$ be a finite set of prime numbers.
Write $\theta = [a_0 ; a_1, a_2 , \ldots ]$ and recall that 
$$
q_{k+1} = a_{k+1} q_k + q_{k-1}, \quad k \ge 2. 
$$
Let $k \ge 2$ and set $d_k = \gcd(q_{k-1}, q_{k+1})$. Since $q_{k-1}$ and $q_k$ are coprime, we see that $d_k$ 
divides $a_{k+1}$. Define
$$
q_{k-1}^* = q_{k-1} / d_k, \quad q_{k+1}^* = q_{k+1} / d_k, \quad a_{k+1}^* = a_{k+1} / d_k.
$$
Then, we have
$$
q_{k+1}^* = a_{k+1}^* q_k + q_{k-1}^*, \quad k \ge 2. 
$$
Let $\eps > 0$. 
By the Schmidt Subspace Theorem, the set of points $(q_{k-1}^*, q_{k+1}^*)$ such that 
$$
q_{k-1}^* q_{k+1}^* \, \prod_{p \in S} |q_{k-1}^* q_{k+1}^* (q_{k+1}^* - q_{k-1}^*)|_p < (q_{k+1}^* )^{- \eps}
$$
is contained in a union of finitely many proper subspaces. 
Since $q_{k-1}^*$ and $q_{k+1}^*$ are coprime, this set is finite. 
We deduce that, for $k$ large enough, we get
$$
\prod_{p \in S} |q_{k-1}^* q_{k+1}^* (q_{k+1}^* - q_{k-1}^*)|_p > 
(q_{k-1}^* q_{k+1}^*)^{-1}  (q_{k+1}^* )^{- \eps},
$$
thus
$$
\prod_{p \in S} |q_{k-1} q_{k+1} (a_{k+1}^* q_{k})|_p > 
(q_{k-1} q_{k+1})^{-1}  (q_{k+1}^* )^{- \eps},
$$
hence
$$
\prod_{p \in S} |q_{k-1} q_{k+1} q_{k}|_p > 
(q_{k-1} q_{k+1})^{-1}  (q_{k+1} )^{- \eps}. 
$$
Recalling that $q_{k-1} < q_k$ and $q_{k+1} < q_k^{\mu - 1 + \eps}$ for $k$ large enough, we get 
$$
[Q_k]_S < q_{k-1} q_{k+1}^{1 + \eps} < q_k^{\mu  + \eps} Q_k^{\eps}
$$
Since
$$
Q_k < q_k^2 q_{k+1} <  q_k^{\mu + 1 + \eps},
$$
we get
$$
[Q_k]_S  <  Q_k^{(\mu  + \eps)/(\mu + 1 + \eps)} Q_k^{\eps}. 
$$
This proves \eqref{eq:EM for q}. The last assertion
can be proved in the same manner thanks to the identity
$p_{k+1}= a_{k+1}p_k + p_{k-1}$ and the inequalities
$\vert p_{k-1}\vert<\vert p_k\vert$
and $\vert p_{k+1}\vert<\vert p_k\vert^{\mu-1+\varepsilon}$
for large $k$.
\end{proof}

The following was suggested at the end of \cite{ErMa39}:

\begin{question}[Erd\H os-Mahler]\label{q:EM}
Let $\theta$ be an irrational real number such that the largest prime factor of $p_n(\theta)q_n(\theta)$ is bounded for infinitely many $n$. Is it true that $\theta$ is a Liouville number?
\end{question}

Erd\H os and Mahler  
stated the existence of $\theta$ with the given properties in Question~\ref{q:EM} without further details. We provide a construction here for the sake of completeness.

Let $S$ and $T$ be disjoint non-empty sets of prime numbers
such that $S$ has at least two elements.    
We construct uncountably    
many $\theta$ such that for infinitely many $n$ the prime factors of $p_n(\theta)$ belong to $S$ while the prime factors of $q_n(\theta)$ belong to $T$. To simplify the notation, we consider the case $S=\{2,3\}$ and $T=\{5\}$. The construction for general $S$ and $T$ follows the same method. The constructed numbers $\theta$ have the form
$$\theta=\sum_{i=1}^{+ \infty}\frac{a_i}{5^{3^i}}.$$

Let $\Gamma$ be the set of positive integers with only prime factors in $\{2,3\}$. For every positive integer $m$, let $\gamma(m)$ denote the smallest element of $\Gamma$ that is greater than $m$. Let
$f(m):=\displaystyle\frac{\gamma(m)-m}{m}$.
By \cite{Tij74}, we have
\begin{equation}
\lim_{m \to + \infty} f(m) = 0.
\end{equation}

First, we construct the sequence of positive integers $s(1)<s(2)<\ldots$ recursively:
\begin{itemize}
\item $s(1)=1$.
\item After having $s(1),\ldots,s(k)$, let $N_k$ be a positive integer depending on $s(k)$ such that
\begin{equation}\label{eq:N_k}
f(m)<\frac{1}{5^{3^{s(k)+1}}}, \ \ \text{for}\ m\geq N_k.
\end{equation}
Then we choose $s(k+1)$ so that:
\begin{equation}\label{eq:s(k+1)}
5^{2\cdot 3^{s(k+1)-1}}\geq N_k\ \text{and}\ s(k+1)>s(k)+1.
\end{equation}
\end{itemize}

Now we construct the $a_i$'s:
\begin{itemize}
\item $a_1=1$.
\item Choose arbitrary $a_i\in\{1,2\}$ for $i\notin\{s(1),s(2),\ldots\}$. Since $s(k+1)>s(k)+1$ for every $k$, the set
$\mathbb{N}\setminus \{s(1),s(2),\ldots\}$ is infinite. Hence there are uncountably   
many choices here.
\item Since $s(1)=1$, we already had $a_{s(1)}$. Suppose we have positive integers $a_{s(1)},\ldots,a_{s(k)}$ with the following properties:
\begin{itemize}
	\item [(i)] For $1\leq j\leq k$, we have $\displaystyle\sum_{i=1}^{s(j)}\frac{a_i}{5^{3^i}}=u_{s(j)}/5^{3^{s(j)}}$ with $u_{s(j)}\in \Gamma$.
	
\smallskip	
	
	\item [(ii)] For $2\leq j\leq k$, we have $\displaystyle\frac{a_{s(j)}}{5^{3^{s(j)}}}< \frac{1}{5^{3^{s(j-1)+1}}}$.
\end{itemize}
We now define $a_{s(k+1)}$ so that the above two properties continue to hold with $j=k+1$ as well. Thanks to property (ii) and the fact that $a_i\leq 2$ for $i\notin\{s(1),s(2),\ldots\}$, we have the rough estimate:
$$\frac{u}{5^{3^{s(k+1)-1}}}:=\sum_{i=1}^{s(k+1)-1} \frac{a_i}{5^{3^i}}\leq \sum_{i=1}^{s(k+1)-1} \frac{2}{5^{3^i}}+\sum_{j=1}^{k-1}\frac{1}{5^{3^{s(j)+1}}}<1$$ 
and hence $u<5^{3^{s(k+1)-1}}$. From
$$\sum_{i=1}^{s(k+1)}\frac{a_i}{5^{3^i}}=\frac{u}{5^{3^{s(k+1)-1}}}+\frac{a_{s(k+1)}}{5^{3^{s(k+1)}}}=\frac{u\cdot 5^{2\cdot 3^{s(k+1)-1}}+a_{s(k+1)}}{5^{3^{s(k+1)}}},$$ 
 we now define $a_{s(k+1)}=\gamma(u\cdot 5^{2\cdot 3^{s(k+1)-1}})-u\cdot 5^{2\cdot 3^{s(k+1)-1}}$. Recall that $\gamma(u\cdot 5^{2\cdot 3^{s(k+1)-1}})$ is the smallest element of $\Gamma$ that is greater than
 $u\cdot 5^{2\cdot 3^{s(k+1)-1}}$. This verifies property (i) for $j=k+1$. To verify (ii) for $j=k+1$, we have:
 \begin{align*}
 \frac{a_{s(k+1)}}{5^{3^{s(k+1)}}}&=\frac{\gamma(u\cdot 5^{2\cdot 3^{s(k+1)-1}})-u\cdot 5^{2\cdot 3^{s(k+1)-1}}}{u\cdot 5^{2\cdot 3^{s(k+1)-1}}}\cdot \frac{u\cdot 5^{2\cdot 3^{s(k+1)-1}}}{5^{3^{s(k+1)}}}\\
 &<\frac{\gamma(u\cdot 5^{2\cdot 3^{s(k+1)-1}})-u\cdot 5^{2\cdot 3^{s(k+1)-1}}}{u\cdot 5^{2\cdot 3^{s(k+1)-1}}}\ \text{since}\ u<5^{3^{s(k+1)-1}}\\
 &=f(u\cdot 5^{2\cdot 3^{s(k+1)-1}})\\
 &<\frac{1}{5^{3^{s(k)+1}}}, \ \text{by \eqref{eq:N_k} and \eqref{eq:s(k+1)}}.  
 \end{align*} 
By the principle of recursive definition, we have $a_i$ for $i\in\{s(1),s(2),\ldots\}$ such that 
property (i) holds for every $j\geq 1$ and property (ii) holds for every $j\geq 2$.
\end{itemize}
Write $\displaystyle u_n/v_n=\sum_{i\leq n}a_i/5^{3^i}$
with $v_n=5^{3^n}$. We have:
\begin{align*}
\vert \theta-u_{s(k)}/v_{s(k)}\vert &=\sum_{i=s(k)+1}^{\infty}\frac{a_i}{5^{3^i}}\\
&< \sum_{i=s(k)+1}^{\infty}\frac{2}{5^{3^i}}+\sum_{j=k}^{\infty}\frac{1}{5^{3^{s(j)+1}}}\\
&<\frac{4}{5^{3^{s(k)+1}}}=\frac{4}{v_{s(k)}^3}. 
\end{align*}
Therefore the $u_{s(k)}/v_{s(k)}$ are among the convergents to $\theta$.

It is not clear to us whether the above numbers $\theta$ are always Liouville numbers. 
However, we suspect that this is the case. 
In order to construct Liouville numbers, we can use a similar method for numbers of the form:
$$\sum_{i=1}^{+ \infty}\frac{b_i}{5^{i!}}.$$

\end{document}